\documentclass[11pt,reqno]{amsart}

\usepackage{amsmath}
\usepackage{amsthm}
\usepackage{graphicx}
\usepackage{amsfonts}
\usepackage{amssymb}
\usepackage{hyperref}
\usepackage{bm}
\usepackage[margin=1.25in]{geometry}
\usepackage{enumerate}
\numberwithin{equation}{section}

\newtheorem{theorem}{Theorem}[section]
\newtheorem{lemma}[theorem]{Lemma}

\newtheorem{corollary}[theorem]{Corollary}

\theoremstyle{definition}
\newtheorem{example}[theorem]{Example}

\newtheorem{remark}[theorem]{Remark}

\newtheorem{innercustomthm}{Assumption}
\newenvironment{assumption}[1]
  {\renewcommand\theinnercustomthm{#1}\innercustomthm}
  {\endinnercustomthm}

\def\R{{\mathbb R}}
\def\N{{\mathbb N}}

\def\P{{\mathcal P}}

\def\W{{\mathcal W}}

\def\tr{{\mathrm{Tr}}}

\def\Var{{\mathrm{Var}}}

\def\transpconst{\gamma}

\title[Quantitative approximate independence for Gibbs measures]{Quantitative approximate independence for continuous mean field Gibbs measures}
\author{Daniel Lacker}
\address{Department of Industrial Engineering \& Operations Research, Columbia University}
\email{daniel.lacker@columbia.edu}
\thanks{This work was partially supported by the Air Force Office of Scientific Research Grant FA9550-19-1-0291.}

\begin{document} 

\begin{abstract}
Many Gibbs measures with mean field interactions are known to be chaotic, in the sense that any collection of $k$ particles in the $n$-particle system are asymptotically independent, as $n\to\infty$ with $k$ fixed or perhaps $k=o(n)$.  This paper quantifies this notion for a class of continuous Gibbs measures on Euclidean space with pairwise interactions, with main examples being systems governed by convex interactions and uniformly convex confinement potentials. The distance between the marginal law of $k$ particles and its limiting product measure is shown to be $O((k/n)^{c \wedge 2})$, with $c$ proportional to the squared temperature. In the high temperature case, this improves upon prior results based on subadditivity of entropy, which yield $O(k/n)$ at best. The bound $O((k/n)^2)$ cannot be improved, as a Gaussian example demonstrates. The results are non-asymptotic, and distance is quantified via relative Fisher information, relative entropy, or the squared quadratic Wasserstein metric. The method relies on an a priori functional inequality for the limiting measure, used to derive an estimate for the $k$-particle distance in terms of the $(k+1)$-particle distance.
\end{abstract}

\maketitle


\section{Introduction}

This paper focuses on continuous Gibbs measures of mean field type on Euclidean spaces, discussed starting in Section \ref{se:mainresults} below.
But the ideas developed in this specific context are suggestive of a more general phenomenon, and so we focus on the latter in this short introduction.
We first recall, with somewhat unconventional terminology, the classical concept of \emph{chaos} formalized in Kac's work on the kinetic theory of gases \cite{kac1956foundations}.

Let $E$ be a Polish space, and let $\P(E)$ denote the space of Borel probability measures on $E$. 
Spaces of probability measures are equipped with the  topology of weak convergence, unless stated otherwise.
For each $n \in \N$ let $P^n$ be an exchangeable probability measure on $E^n$, and let $P^n_k$ denote its marginal on $E^k$ for $k < n$. For $\mu \in \P(E)$, we say that \emph{$(P^n)_{n \in \N}$ is locally $\mu$-chaotic} if
\begin{align}
\lim_{n\to\infty}P^n_k = \mu^{\otimes k}, \quad \text{in } \P(E^k), \ \text{for each } k \in \N, \label{intro:def:localchaos}
\end{align}
where $\mu^{\otimes k}$ denotes the $k$-fold product measure.
Let $L_n : E^n \to \P(E)$ denote the empirical measure map, $L_n(x_1,\ldots,x_n) := \frac{1}{n}\sum_{i=1}^n\delta_{x_i}$.
We say that \emph{$(P^n)_{n \in \N}$ is globally $\mu$-chaotic} if
\begin{align}
\lim_{n\to\infty} P^n \circ L_n^{-1}  = \delta_\mu, \quad \text{in } \P(\P(E)).\label{intro:def:globalchaos}
\end{align}
Local and global chaos are well known to be equivalent; that is, \eqref{intro:def:localchaos} holds if and only if \eqref{intro:def:globalchaos} holds, and we may then say more simply that \emph{$(P^n)_{n \in \N}$ is $\mu$-chaotic}.
However, \emph{quantitative} forms of this \emph{qualitative} equivalence are not straightforward or canonical.
For this reason,  it will help to distinguish them via the terms \emph{local}, dealing only with boundedly many coordinates (or, later, $o(n)$ coordinates), and \emph{global}, dealing with the full $n$-particle system or its empirical measure.

Recent years have seen substantial progress on quantifying chaos in various contexts, most notably its propagation along the dynamics of various interacting particle systems.
Most existing methods are fundamentally global in nature, with local estimates deduced from global ones.
One common strategy for passing from global to local rates of chaos, appearing in many previous papers such as \cite{benarous-zeitouni,jabin-wang-bounded,jabin-wang-W1inf,jabir2019rate}, is to use the well known subadditivity  inequality for relative entropy defined in \eqref{def:entropy},
\begin{align}
H(P^n_k\,|\, \mu^{\otimes k}) \le \frac{2k}{n} H(P^n \,|\, \mu^{\otimes n}). \label{def:intro:subadditivity}
\end{align}
See \cite[Lemma 3.9]{del2001genealogies}. 
Then, by proving a global estimate of the form
\begin{align}
H(P^n\,|\,\mu^{\otimes n}) = O(1), \label{intro:globalest1}
\end{align}
one can immediately deduce the local estimate
\begin{align}
H(P^n_k\,|\,\mu^{\otimes k}) = O(k/n). \label{intro:localrate1}
\end{align}
See Section \ref{se:global-to-local} for further discussion, including an alternative strategy proposed in \cite{hauray2014kac}.

The main results of the paper, stated in Section \ref{se:mainresults}, show that the local estimate \eqref{intro:localrate1} can be improved to $O((k/n)^2)$ for a class of continuous Gibbs measures arising as invariant measures of interacting diffusions, at least at high temperature. 
We also show that this local estimate is optimal, via explicit computations for a Gaussian example given in Section \ref{se:gaussian}. 
We make no claims about improving the global estimate \eqref{intro:globalest1}, which is typically impossible.
This suggests what we suspect to be a more general phenomenon, that \emph{local rates deduced from global rates are often suboptimal}. The companion paper  \cite{lacker-dynamic} develops similar results and techniques for dynamic models.

From the perspective of physical applications, the class of Gibbs measures covered by our results is somewhat limited, consisting essentially of interaction functions having either bounded or Lipschitz gradient. In particular, we do not treat singular interactions or discrete models.
Moreover, our results apply only in settings with unique equilibria (no phase transition).
The novelty in this work, rather, is in obtaining optimal quantitative bounds.
In addition, we expect the method to be more broadly applicable.

The following Section \ref{se:mainresults} presents the setting and main results, along with a discussion of related literature in Section \ref{se:relatedliterature}.
The remaining sections are devoted to proofs.

\section{Main results} \label{se:mainresults}

Let $d \in \N$. We study exchangeable probability measures on $(\R^d)^n$ of the form
\begin{align}
P^n(dx) = \frac{1}{Z_n}\exp\bigg( - \frac{\beta}{ n-1 }\sum_{1 \le i < j \le n}V(x_i,x_j)  \bigg)\,\lambda^{\otimes n}(dx), \label{def:Pn}
\end{align}
where $Z_n > 0$ is a normalization constant.
Here $\beta$ is the inverse temperature, $\lambda$ is the reference measure, and $V$ is the  interaction function.

\begin{assumption}{\textbf{A}} \label{assumption:A}
We are given $\beta > 0$, $\lambda \in \P(\R^d)$, and $V : \R^d \times \R^d \to \R$ satisfying:
\begin{itemize}
\item $\lambda$ is absolutely continuous with respect to Lebesgue measure.
\item $V$ is bounded from below and symmetric, meaning $V(x,y)=V(y,x)$ for all $x,y$.
\item The weak gradient $\nabla_1V$ of $V$ in its first argument belongs to $L^1_{\mathrm{loc}}(\R^d \times \R^d)$.
\end{itemize}
\end{assumption}

Assumption \ref{assumption:A} is more than enough to ensure that $P^n$ is well-defined. The boundedness of $V$ from below simplifies the exposition but could be sharpened. Of course, $\beta$ may be absorbed into $V$, but we prefer to keep it separate to illustrate its role more clearly.
The following strengthening of Assumption \ref{assumption:A} is an important special case:

\begin{assumption}{\textbf{B}} \label{assumption:B}
We have $\beta > 0$ and $\lambda(dx)=e^{-\beta U(x)}dx$ for some function $U$.
The function $V$ takes the form $V(x,y)=V(x-y)$, and the functions $U,V : \R^d \to \R$ are even. There exist $\kappa, L > 0$ such that $\nabla^2 U(x) \ge \kappa I$ and $0 \le \nabla^2 V(x) \le L I$ in the sense of positive definite order, for all $x \in \R^d$.
\end{assumption}

The evenness of $(U,V)$ in Assumption \ref{assumption:B} is convenient for ensuring that $P^n$ is centered, but it could certainly be relaxed.

The $n\to\infty$ behavior of $P^n$ is well understood, especially in the regime we consider where $\beta$ is fixed.
The limiting behavior is described in terms of solutions $\mu \in \P(\R^d)$ of the fixed point equation
\begin{align}
\mu(dx) &= \frac{1}{Z} \exp\big( -\beta \langle \mu,V(x,\cdot)\rangle \big)\,\lambda(dx), \quad Z > 0. \label{def:mu-fixedpoint}
\end{align}
When there is a unique $\mu$ satisfying \eqref{def:mu-fixedpoint}, one expects $(P^n)_{n \in \N}$ to be $\mu$-chaotic.
There are other perspectives on $P^n$ and $\mu$ which are worth mentioning but which will not be used. These formulations are valid under Assumption \ref{assumption:B}, and in many more general cases in which $\lambda(dx)=e^{-\beta U(x)}dx$ with $(U,V)$ sufficiently nice.
Then, $P^n$ is the invariant measure of the Markov process $(X^1,\ldots,X^n)$ governed by the stochastic differential equation (SDE)
\begin{align}
dX^i_t = -\beta\bigg(\nabla U(X^i_t) + \frac{1}{n-1}\sum_{j \neq i} \nabla_1 V(X^i_t,X^j_t) \bigg)dt + \sqrt{2}\, dB^i_t, \quad i=1,\ldots,n, \label{intro:SDE}
\end{align}
where $B^1,\ldots,B^n$ are independent $d$-dimensional Brownian motions. 
The fixed point equation \eqref{def:mu-fixedpoint} characterizes invariant measures for the corresponding  McKean-Vlasov SDE
\begin{align}
dX_t = -\beta\left(\nabla U(X_t) + \langle \mu_t,\nabla_1 V(X_t,\cdot)\rangle\right)dt + \sqrt{2}\,dB_t, \quad \mu_t=\mathrm{Law}(X_t). \label{intro:MVSDE}
\end{align}
For yet another perspective, it is known that $(P^n \circ L_n^{-1})_{n \in \N} \subset \P(\P(\R^d))$ satisfies a large deviation principle with good rate function $J - \inf J$, where
\begin{align}
J(\nu) :=  \beta \langle \nu^{\otimes 2}, V\rangle + H(\nu\,|\,\lambda), \quad \nu \in \P(\R^d) \label{def:ratefunction}
\end{align}
and where $H$ denotes the relative entropy, defined in \eqref{def:entropy}. The solutions of \eqref{def:mu-fixedpoint} correspond to the minimizers of $J$.
We elaborate on this last perspective in Section \ref{se:MFGibbs}.

\begin{remark}
Our results below will \emph{assume} the existence of $\mu$ satisfying \eqref{def:mu-fixedpoint}, which is well understood under various hypotheses on $(\beta,\lambda,V)$. Existence can be approached directly, e.g., by applying Schauder's fixed point theorem as in \cite[Section 4]{benachour1998-I}. Alternatively, after rigorously linking them to \eqref{def:mu-fixedpoint}, the SDE \eqref{intro:MVSDE} or the variational problem \eqref{def:ratefunction} can be used to prove existence and sometimes uniqueness \cite{carrillo2003kinetic,cattiaux2008probabilistic}. For instance, existence and uniqueness for \eqref{def:mu-fixedpoint} follows from \cite[Theorem 2.1]{carrillo2003kinetic}, under Assumption \ref{assumption:B}.
\end{remark}

A probability measure on a product space $E^n$ is said to be \emph{exchangeable} if it is invariant with respect to coordinate permutations.
Clearly $P^n$ is exchangeable, so any group of $k$ coordinates share the same marginal law, denoted $P^n_k \in \P((\R^d)^k)$.
We will use various notions of distance to compare $P^n_k$ with $\mu^{\otimes k}$, where $\mu$ is a solution of \eqref{def:mu-fixedpoint}.
The  relative entropy between any two probability measures $(\nu,\nu')$ on the same measurable space is defined by
\begin{align}
H(\nu\,|\,\nu') = \int \frac{d\nu}{d\nu'}\log\frac{d\nu}{d\nu'}\,d\nu', \ \ \text{ if } \nu \ll \nu', \ \text{ and } \ \  H(\nu\,|\,\nu')= \infty \ \ \text{otherwise}. \label{def:entropy}
\end{align}
For a bounded signed measure $\nu$ the total variation is given by
\begin{align*}
\|\nu\|_{\mathrm{TV}} := \sup_{|f| \le 1} \langle \nu,f\rangle,
\end{align*}
where the supremum is over measurable functions $f$ with $|f| \le 1$.
For probability measures $(\nu,\nu')$ on a Euclidean space, the relative Fisher information is defined by
\begin{align*}
I(\nu\,|\,\nu') = 
\begin{cases}
\int \left|\nabla \log \frac{d\nu}{d\nu'}\right|^2\,d\nu &\text{ if } \nu \ll \nu', \ \text{ and } \nabla \log \frac{d\nu}{d\nu'} \text{ exists in } L^2(\nu) \\
\infty &\text{otherwise},
\end{cases}
\end{align*}
and the $p$-Wasserstein distance for $p \ge 1$ is given by
\begin{align*}
\W_p^p(\nu,\nu') := \inf_\pi \int |x-y|^p\,\pi(dx,dy),
\end{align*}
where the infimum is over all couplings of $(\nu,\nu')$, and $|\cdot|$ denotes the usual Euclidean norm.

Our first theorem estimates $I(P^n_k \,|\,\mu^{\otimes k})$, under a smallness-type assumption for $\nabla_1 V$.

\begin{theorem} \label{th:main}
Suppose $(\beta,\lambda,V)$ satisfies Assumption \ref{assumption:A}.
Define $P^n \in \P((\R^d)^n)$ by \eqref{def:Pn}, and suppose there exists $\mu \in \P(\R^d)$ satisfying \eqref{def:mu-fixedpoint}.
Assume the following:
\begin{enumerate}[(1)]
\item Square-integrability:
\begin{align}
M := \int_{(\R^d)^2} \left|  \nabla_1 V(x,y) - \langle \mu,\nabla_1V(x,\cdot)\rangle \right|^2 \, P^n_2 (dx,dy) < \infty. \label{asmp:static-moment}
\end{align}
\item Transport-type inequality: There exists $0 < \transpconst < \beta^{-2}$ such that
\begin{align}
|\langle \mu - \nu, \nabla_1 V(x,\cdot)\rangle|^2 \le  \transpconst I(\nu\,|\,\mu), \ \ \forall x \in \R^d, \ \nu \in \P(\R^d), \label{asmp:static-transp}
\end{align}
\end{enumerate}
Then, for integers $n > k > 1$, we have 
\begin{align}
I(P^n_k\,|\,\mu^{\otimes k}) &\le  k M \beta^2 \left(  C  \frac{\sqrt{k-1}}{n-1} + (\beta\sqrt{\transpconst})^{n-k} \right)^2, \label{eq:static-result}
\end{align}
for a constant $C > 0$ satisfying
\begin{align}
C \le \frac{ 8\pi  }{ \log \frac{1}{\beta^2\transpconst} } \left( \frac{  1 + \beta^2\transpconst }{(1-\beta\sqrt{\transpconst})^2  } + \frac{4}{e \log \frac{1}{\beta^2\transpconst}}\right).  
\label{const:static}
\end{align}
\end{theorem}

In particular, if the constants $M$ and $\transpconst$ in Theorem \ref{th:main} can be bounded independently of $n$, we obtain $I(P^n_k\,|\,\mu^{\otimes k}) = O((k/n)^2)$.
This proves a form of \emph{increasing propagation of chaos} as in \cite{benarous-zeitouni}, in the sense that $I(P^n_k \,|\,\mu^{\otimes k})\to 0$  as $n\to\infty$ even if $k$ diverges, as long as $k=o(n)$. 
The constant $C$ was not optimized, but note that it goes to zero as $\beta^2\gamma \to 0$.

\begin{example} \label{ex:gaussian-static}
A Gaussian example shows that the rate $(k/n)^2$ in Theorem \ref{th:main} is optimal.
Consider the case $\beta=1$, $d=1$, $V(x,y)=b(x-y)^2/2$, and $\lambda=N(0,1/a)$, where $a,b > 0$. Then $P^n$ is Gaussian, and $\mu=N(0,1/(a+b))$ can be shown to be the unique solution of \eqref{def:mu-fixedpoint}. When $n\to\infty$ and $k=o(n)$, we show in Section \ref{se:gaussian} that $(n/k)^2 \W_2^2(P^n_k,\mu^{\otimes k})$ is bounded away from zero.
By Talagrand's inequality  and the log-Sobolev inequality, this implies the same for $(n/k)^2 H(P^n_k\,|\,\mu^{\otimes k}) > 0$ and $(n/k)^2 I(P^n_k\,|\,\mu^{\otimes k}) > 0$. (Actually, all of these quantities can be computed explicitly.)
To see that this case fits the assumptions of Theorem \ref{th:main}, note (again using Talagrand's inequality and log-Sobolev) that $\mu$ satisfies 
\begin{align*}
\W_2^2(\mu,\nu) \le \frac{2}{a+b} H(\nu\,|\,\mu) \le \frac{1}{(a+b)^2}I(\nu\,|\,\mu), \qquad \forall \nu \in \P(\R).
\end{align*}
Note also that $\nabla_1 V(x,y) = b(x-y)$, and so Kantorovich duality implies
\begin{align*}
|\langle \nu-\mu,\nabla_1 V(x,\cdot)\rangle|^2 \le b^2 \W_1^2(\mu,\nu) \le \frac{b^2}{(a+b)^2}I(\nu\,|\,\mu), \qquad \forall x \in \R, \ \nu \in \P(\R).
\end{align*}
The assumptions of Theorem \ref{th:main} hold here with $\transpconst = b^2/(a+b)^2 < 1 = \beta^{-2}$.
\end{example}

\begin{remark} \label{re:uniqueness}
As mentioned above, the existence of $\mu \in \P(\R^d)$ satisfying the fixed point equation \eqref{def:mu-fixedpoint} is taken as an assumption in Theorem \ref{th:main}. In fact, the transport-type inequality \eqref{asmp:static-transp} implies that there can be no other solution, up to integrability. Indeed, if $\widetilde\mu$ is another solution, then
\begin{align*}
I(\widetilde\mu\,|\,\mu) &= \beta^2 \int_{\R^d}\left| \langle \widetilde\mu - \mu, \nabla_1V(x,\cdot)\rangle\right|^2\,\widetilde\mu(dx) \le \transpconst \beta^2 I(\widetilde\mu\,|\,\mu).
\end{align*}
If $|\nabla_1 V| \in L^2((\widetilde\mu +\mu) \otimes  \widetilde\mu )$, which can easily be established \emph{a priori} in many cases, then $I(\widetilde\mu\,|\,\mu) < \infty$, and the assumption $\transpconst \beta^2 < 1$ implies that $\widetilde\mu=\mu$.
\end{remark}

The most difficult assumption to check in Theorem \ref{th:main} is typically the transport-type inequality \eqref{asmp:static-transp}, and the crucial assumption that   $\beta^2 < 1/\transpconst$ can be interpreted as a \emph{high-temperature} condition. It is well known that the invariant measure of a McKean-Vlasov SDE often fails to be unique at low temperature \cite{dawson1983critical,herrmann2010non}, and this would be an obvious impediment to our result. Related smallness conditions are pervasive, at least in the absence of ample convexity, in the literature on uniform-in-time propagation of chaos \cite{guillin2019uniform,durmus2020elementary}. 

Natural corollaries of Theorem \ref{th:main} arise in settings where a \emph{transport-information inequality} holds. That is, suppose there exist $L,\transpconst'>0$ and a metric $\rho$ on $\R^d$ such that $\nabla V_1 (x,\cdot)$ is $L$-Lipschitz with respect to $\rho$ for every $x$, and
\begin{align}
\W_{1,\rho}^2(\nu,\mu) \le \transpconst' I(\nu \,|\,\mu), \quad \forall \nu \in \P(\R^d), \label{def:transp-ineq-gen}
\end{align}
where the 1-Wasserstein distance $\W_{1,\rho}$ is defined relative to $\rho$.  Then \eqref{asmp:static-transp} holds with constant $\transpconst' L^2$.
Transport-information inequalities of this form were studied in \cite{guillin2009transportation} for characterizing concentration inequalities for Markov processes, and we refer to \cite{guillin2009transportation,liu2017new} for further details and tractable sufficient conditions.
Using some of their ideas, we will show that inequality \eqref{def:transp-ineq-gen} can be verified in two noteworthy cases. Essentially, Corollary \ref{co:static-bounded} uses \eqref{def:transp-ineq-gen} with the trivial metric $\rho(x,y)=1_{\{x \neq y\}}$, and Corollary \ref{co:static-convex} uses \eqref{def:transp-ineq-gen} with the usual Euclidean metric $\rho(x,y)=|x-y|$.

\begin{corollary} \label{co:static-bounded}
Suppose $(\beta,\lambda,V)$  satisfies Assumption \ref{assumption:A}, and assume $L :=  \| |\nabla_1 V|\|_\infty < \infty$.
Define $P^n$ by \eqref{def:Pn}, and suppose $\mu \in \P(\R^d)$ satisfies \eqref{def:mu-fixedpoint}. Assume there exists $c_\mu < (\beta L)^{-2}$ such that the Poincar\'e inequality holds:
\begin{align}
\Var_\mu(f) \le c_\mu \langle \mu,|\nabla f|^2\rangle, \quad \text{ for } f : \R^d\to\R \text{ such that } \nabla f \text{ exists in } L^2(\mu). \label{asmp:co:static-poincare}
\end{align}
Then the conclusion of Theorem \ref{th:main} holds with $\transpconst := c_\mu L^2$ and $M := 2L^2$.
Moreover, if $1 \le k \le n$ and $\W_{1,H}$ denotes the 1-Wasserstein distance on $(\R^d)^k$ defined using the Hamming metric
$((x_1,\ldots,x_k),(y_1,\ldots,y_k)) \mapsto \sum_{i=1}^k 1_{\{x_i \neq y_i\}}$,
then we have
\begin{align}
\W_{1,H}^2(P^n_k,\mu^{\otimes k}) &\le kc_\mu I(P^n_k\,|\,\mu^{\otimes k}). \label{ineq:hamming1}
\end{align}
\end{corollary}

\begin{remark}
There are various easy sufficient conditions for the Poincar\'e inequality \eqref{asmp:co:static-poincare}:
\begin{itemize}
\item If Assumption \ref{assumption:B} holds, then \eqref{asmp:co:static-poincare} holds with $c_\mu = 1/\beta\kappa$ (see Section \ref{se:BakryEmery}).
\item If $\lambda$ satisfies the Poincar\'e inequality with constant $c_\lambda$, and if $\mathrm{osc}(V) := \sup V - \inf V$ is finite, then \eqref{asmp:co:static-poincare} holds with $c_\mu = c_\lambda e^{-2\mathrm{osc}(V)}$ (see \cite[Property 2.6]{guionnet2003lectures}).
\item If $\lambda$ is the uniform measure on a bounded connected set and $V$ is $C^1$, then \eqref{asmp:co:static-poincare} holds for some $c_\mu$. This is a special case of the previous condition.
\end{itemize}
\end{remark}

Our second example covers the case of convex potentials, for which the famous Bakry-Emery criterion makes it fairly straightforward to check the hypotheses of Theorem \ref{th:main}.

\begin{corollary} \label{co:static-convex}
Suppose $(\beta,\lambda,V)$  satisfies Assumption \ref{assumption:B}. Then there exists a unique $\mu \in \P(\R^d)$ satisfying \eqref{def:mu-fixedpoint}. Define $P^n$ by \eqref{def:Pn}.
If $L < \kappa$, then for integers $n > k > 1$, we have 
\begin{align*}
\W_2^2(P^n_k\,|\,\mu^{\otimes k}) &\le \frac{2}{\beta \kappa} H(P^n_k\,|\,\mu^{\otimes k}) \le \frac{1}{\beta^2 \kappa^2}I(P^n_k\,|\,\mu^{\otimes k} ), \quad \text{and} \\
I(P^n_k\,|\,\mu^{\otimes k}) &\le k \frac{\beta L^2d}{\kappa}  \left(  C  \frac{\sqrt{k-1}}{n-1} + \left(\frac{ L}{\kappa} \right)^{n-k} \right)^2,
\end{align*}
for a constant $C > 0$ satisfying
\begin{align*}
C &\le \frac{ 4\pi }{(1-(L/\kappa))^2\log (\kappa/L)} \left(  \frac{ 1 + (L/\kappa)^2 }{(1-(L/\kappa))^2  } + \frac{2}{e \log (\kappa/L)} \right).
\end{align*}
\end{corollary}

Note by Pinsker's inequality that the estimate $H(P^n_k\,|\,\mu^{\otimes k})=O((k/n)^2)$ of Corollary \ref{co:static-convex} implies also
\begin{align}
\|P^n_k - \mu^{\otimes k}\|_{\mathrm{TV}} = O(k/n). \label{ineq:newTV}
\end{align}

\subsection{Reversed relative entropy}
By reversing the order of arguments in the relative entropy, we obtain a result for all temperatures in Theorem \ref{th:main-rev} below, in contrast with Theorem \ref{th:main} which is limited to small $\beta$. The price, however, is that Theorem \ref{th:main-rev} assumes functional inequalities for $P^n$ \emph{and its conditional measures}, which are more difficult to check.  
In the following, let $P^n_{k+1|k}(x_1,\ldots,x_k)$ denote a version of the conditional law of $X_{k+1}$ given $(X_1,\ldots,X_k)=(x_1,\ldots,x_k)$, when $(X_1,\ldots,X_n)$ has law $P^n$.

\begin{theorem} \label{th:main-rev}
Suppose $(\beta,\lambda,V)$ satisfies Assumption \ref{assumption:A}.
Define $P^n$ by \eqref{def:Pn}, and suppose $\mu \in \P(\R^d)$ satisfies \eqref{def:mu-fixedpoint}.
Assume the following:
\begin{enumerate}[(1)]
\item Square-integrability:
\begin{align}
M := \int_{(\R^d)^2} \left|  \nabla_1 V(x,y) - \langle \mu,\nabla_1 V(x,\cdot)\rangle \right|^2 \, \mu^{\otimes 2} (dx,dy) < \infty. \label{asmp:static-moment-rev}
\end{align}
\item Transport-type inequality: There exists $\transpconst < \infty$ such that
\begin{align}
|\langle \mu - P^n_{k+1|k}(x_1,\ldots,x_k), \nabla_1 V(x_1,\cdot)\rangle|^2 \le   \transpconst H(\mu\,|\,P^n_{k+1|k}(x_1,\ldots,x_k)), \label{asmp:static-transp-rev}
\end{align}
for all $1\le k < n$ and $x_1,\ldots,x_k \in \R^d$.
\item Log-Sobolev inequality: There exists $\eta < \infty$ such that, for each $1 \le k \le n$,
\begin{align}
H(\mu^{\otimes k}\,|\,P^n_k) \le \eta I(\mu^{\otimes k}\,|\,P^n_k), \qquad \forall \nu \in \P((\R^d)^k). \label{asmp:static-LSI-rev}
\end{align}
\end{enumerate}
Let $\epsilon > 0$ be such that
\begin{align*}
\alpha := \eta\transpconst\beta^2(1+\epsilon)   \neq 1/2.
\end{align*}
Then, for integers $n > k \ge 1$, we have
\begin{align*}
H(\mu^{\otimes k} \,|\, P^n_k) &\le \left(\frac{ (1+2\alpha)^{2 \vee (1/\alpha)} }{\epsilon \transpconst\beta^2 \alpha |1-2\alpha| }  +  2\eta M\beta^2 \right)\left(\frac{1 + \alpha k}{1 + \alpha (n-1)}\right)^{2 \wedge(1/\alpha)}.
\end{align*}
\end{theorem}

Note that if $\alpha=1/2$ then we may choose $\epsilon$ smaller to make $\alpha < 1/2$.
If $n \to\infty$ and $k=o(n)$, and if $(M,\transpconst,\eta)$ stay bounded, then Theorem \ref{th:main-rev} and Pinsker's inequality yield
\begin{align*}
\|  P^n_k - \mu^{\otimes k}\|_{\mathrm{TV}}^2 \le 2H( \mu^{\otimes k} \,|\, P^n_k ) = \begin{cases} O((k/n)^{\frac{1}{\eta\transpconst \beta^2} - \delta}) &\text{for all } \delta > 0 \text{ if } \eta\transpconst\beta^2 \ge 1/2 \\ 
O((k/n)^2) &\text{if } \eta\transpconst\beta^2 < 1/2. \end{cases}
\end{align*}
In the high-temperature case $\eta\transpconst\beta^2 < 1/2$, the rate $(k/n)^2$ cannot be improved. But it is not clear if the exponent $1/\eta\transpconst \beta^2$ is  optimal in the low-temperature case $\eta\transpconst\beta^2 \ge 1/2$.
In either case, we have $H( \mu^{\otimes k} \,|\, P^n_k ) \to 0$ as long as $k = o(n)$.

\begin{remark}
Note that the functional inequalities in assumptions (2) and (3) Theorem \ref{th:main-rev} are required to hold only at $\mu$, rather than an arbitrary measure. However, in every example we considered, it is simpler to check more general criteria which ensure that conditions (2) and (3) hold with $\mu$ and $\mu^{\otimes k}$ replaced by arbitrary $\nu \in \P(\R^d)$ and $\nu \in \P((\R^d)^k)$, respectively.
On a related note, Theorem \ref{th:main} remains true, with the same proof, if the the assumption (2) therein is required to hold only for $\nu\in \{P^n_{k+1|k}(x_1,\ldots,x_k) : k < n, \, x_1,\ldots,x_k \in \R^d\}$.
\end{remark}

The assumption (2) in Theorem \ref{th:main-rev} is a strong one, requiring a transport inequality holding uniformly over a family of conditional measures.
It holds in the case $\nabla_1 V$ is bounded by Pinsker's inequality, but also in the strongly convex setting:

\begin{corollary} \label{co:static-convex-rev}
Suppose $(\beta,\lambda,V)$  satisfies Assumption  \ref{assumption:B}.
Then there exists a unique $\mu \in \P(\R^d)$ satisfying \eqref{def:mu-fixedpoint}.
Define $P^{(n)}$ by \eqref{def:Pn}.
Let $\epsilon > 0$ be such that 
\begin{align*}
\alpha := (1+\epsilon)(L/\kappa)^2 \neq 1/2.
\end{align*}
Then, for integers $n > k \ge 1$, we have
\begin{align*}
H(\mu^{\otimes k} \,|\, P^n_k) &\le C \left(\frac{1 + \alpha k}{1 + \alpha (n-1)}\right)^{2 \wedge (1/\alpha)},
\end{align*}
for a constant $C > 0$ satisfying
\begin{align*}
C \le \left(1+\frac{1}{\epsilon}\right)\frac{ (1+2\alpha)^{2 \vee (1/\alpha)} }{ 2\beta\kappa  \alpha^2 |1-2\alpha| } + \alpha d .
\end{align*}
\end{corollary}

We prove Corollary \ref{co:static-convex-rev} by checking that $P^n$ satisfies a Bakry-Emery curvature condition (i.e., that the density is uniformly log-concave) uniformly in $n$, which was essentially due to \cite{malrieu2001logarithmic}. This curvature condition is preserved under conditioning and marginalizing, which allows us to check (2) and (3).
The recent paper \cite{guillin2019uniform} establishes uniform (in $n$) log-Sobolev and Poincar\'e inequalities beyond the convex setting, which are not directly applicable here but similarly involve checking functional inequalities for conditional measures.

\subsection{Related literature} \label{se:relatedliterature}

The companion paper \cite{lacker-dynamic} develops similar ideas in the dynamic context. Precisely, suppose we initialize the SDE system \eqref{intro:SDE} at some distribution $P^n(0)$, and let $P^n(t)$ denote the law at time $t$ of the solution. Similarly, let $\mu(t)$ denote the time-$t$ law of a solution of \eqref{intro:MVSDE}, initialized from some $\mu(0)$. Using an entropic iteration technique similar to that of this paper, it is shown in \cite[Theorem 2.2]{lacker-dynamic} that
\begin{align*}
H(P^n_k(0)\,|\,\mu(0)^{\otimes k}) = O((k/n)^2) \quad \Longrightarrow \quad H(P^n_k(t)\,|\,\mu(t)^{\otimes k}) = O((k/n)^2), \ \ \forall t > 0,
\end{align*}
again with explicit constants. A similar result is shown when the order of arguments is reversed in the relative entropy. This is the \emph{optimal} rate of propagation of chaos in this context, in the sense that we cannot improve the $O((k/n)^2)$ in general, even for i.i.d.\ initializations $P^n(0)=\mu(0)^{\otimes n}$. The results of the present paper provide broad classes of non-i.i.d.\ examples of $P^n(0)$ such that $H(P^n_k(0)\,|\,\mu(0)^{\otimes k})$ or $H(\mu(0)^{\otimes k}\,|\,P^n_k(0))$ is $O((k/n)^2)$.
See \cite{lacker-dynamic} for details and related literature on dynamic models.

\subsubsection{Mean field Gibbs measures} \label{se:MFGibbs}

There is an extensive literature on mean field Gibbs measures of much more general forms than \eqref{def:Pn}. For instance, consider a Polish space $E$, a functional $F : \P(E) \to \R$, and a reference measure $\lambda \in \P(E)$. Consider the Gibbs measure  $dQ^n/d\lambda^{\otimes n} = e^{-nF \circ L_n}/Z_n$, where $L_n : E^n \to \P(E)$ is again the empirical measure. For $F$ continuous and bounded from below, Sanov's theorem and Varadhan's lemma imply that $Q^n \circ L_n^{-1}$ satisfies a large deviation principle with good rate function
\begin{align*}
\P(E) \ni \nu \mapsto F(\nu) + H(\nu\,|\,\lambda) - \inf_{\nu' \in \P(E)}(F(\nu')+ H(\nu'\,|\,\lambda)).
\end{align*}
For pairwise interactions, i.e., $F(\nu) = \langle\nu^{\otimes 2},V\rangle$ for some $V : E^2 \to \R$, this large deviation principle is now known to hold under much broader assumptions which cover singular interactions $V$; these results require much more care, particularly when the temperature is allowed to depend on $n$  \cite{chafai2014first,dupuis2020large,liu2020large,garcia2019large}.
When $F + H(\cdot\,|\,\lambda)$ admits a unique minimizer $\mu$, the large deviation principle implies the law of large numbers $Q^n \circ L_n^{-1} \to \delta_\mu$, which in turn implies the local chaos $Q^n_k \to \mu^{\otimes k}$ as $n\to\infty$ for each fixed $k$. It was shown in \cite[Theorem 1]{benarous-zeitouni}, using large deviation techniques, that this can be made more quantitative when $F$ is a finite-range interaction. In particular, if there is a unique minimizer $\mu$ which is moreover \emph{non-degenerate} in a suitable sense, then $H(P^n\,|\,\mu^{\otimes n}) = O(1)$ by \cite[Theorem 1]{benarous-zeitouni}. By subadditivity \eqref{def:intro:subadditivity}, this implies $H(P^n_k \,|\,\mu^{\otimes k}) = O(k/n)$.
See also \cite[Theorem 2]{benarous-zeitouni} for the case of finitely many non-degenerate minimizers.
These large deviations techniques, however, are  \emph{global} in nature and do not appear to be capable of producing the non-asymptotic  $O((k/n)^2)$ \emph{local} estimates given in  Theorem \ref{th:main} or \ref{th:main-rev}.

It would be interesting to try to adapt our approach beyond the continuous, Euclidean setting, for instance to cover discrete spin systems.
This might involve using Dirichlet forms to replace the Fisher information which is central to our framework, but it is far from clear how this would proceed.

\subsubsection{Finite de Finetti-type theorems}

Quantitative \emph{local chaos} can be interpreted as describing ``how approximately i.i.d." a measure is. The optimal estimate $\|P^n_k-\mu^{\otimes k}\|_{\mathrm{TV}} = O(k/n)$ obtained in \eqref{ineq:newTV} appears to be new for Gibbs measures, though it is known for other specific chaotic sequences.
The first famous example is when $P^n$ is the uniform measure on the sphere of radius $\sqrt{n}$ in $\R^n$, which is well known to be $\gamma$-chaotic for the standard Gaussian measure $\gamma$ on $\R$. It was shown by Diaconis-Freedman \cite[Theorem 1]{diaconis1987dozen} that $\|P^n_k-\gamma^{\otimes k}\|_{\mathrm{TV}} \le 2(k+3)/(n-k-3)$ for $n \ge k + 4$. When the uniform measure $P^n$ is on the simplex instead of the sphere (i.e., on the sphere in $\ell^1$ instead of $\ell^2$), a similar $O(k/n)$ estimate holds with $\gamma$ replaced by the exponential distribution \cite[Theorem 2]{diaconis1987dozen}.
Extensions to $\ell^p$-spheres were developed in \cite{rachev1991approximate,mogul1991finetti,naor2003projecting}, 
and very recently to Orlicz balls \cite{johnston2020maxwell}.
Generalizing in a different direction, earlier work of Diaconis-Freedman \cite{diaconis1988conditional} gives an $O(k/n)$ local chaos estimate for certain i.i.d.\ random vectors conditioned to the sphere, and see \cite[Section 4]{hauray2014kac} and \cite[Section 4]{carlen2008entropy} for entropic perspectives on similar questions.

Also relevant here is yet another theorem of Diaconis-Freedman, on approximating finite exchangeble sequences by mixtures of product measures, essentially quantifying the famous theorems of de Finetti and Hewitt-Savage. They show in \cite{diaconis1980finite} that any symmetric probability measure $P^n$ on $E^n$, where $E$ is a finite set, satisfies
\[
\Big\|P^n_k-\int_{\P(E)} m^{\otimes k}\,\rho_n(dm)\Big\|_{\mathrm{TV}} \le 2|E|k/n,
\]
for some $\rho_n \in \P(\P(E))$.
If $E$ is infinite, they give instead a bound of $k(k-1)/n$. Each bound is optimal absent further assumptions.
In fact, one can always take $\rho_n = P^n \circ L_n^{-1}$ for the latter bound;
underlying this argument is the observation that $P^n_k$ represents the law of $k$ particles chosen at random \emph{without replacement}, whereas $\int_{E^n} L_n^{\otimes k}\,dP^n$ represents the law of $k$ particles chosen at random \emph{with replacement}.
One consequence of our main results is an identification of certain classes of finite exchangeable measures on continuous state space for which this $k(k-1)/n$ bound may be improved to $O(k/n)$, matching the finite state space case. In our case, of course, we are dealing with exchangeable measures that are in fact approximately i.i.d., and so our ``mixing measure" $\rho_n$ is taken to be $\delta_{\mu}$.

Recent work in random matrix theory has explored similar quantitative approximate independence properties for various natural models. Using explicit forms of the densities,  \cite{diaconis1992finite} gives non-asymptotic estimates of the total variation distance between a $p \times q$ matrix of i.i.d.\ Gaussians and the upper-left $p \times q$ submatrix of a uniformly random $n \times n$ orthogonal matrix.
In a similar spirit, the recent work \cite{bubeck2018entropic} bounds the total variation distance between Wishart and Wigner matrices of high dimension, via an inductive argument relying on the chain rule for relative entropy which is not terribly distant from our method.

\subsubsection{Local propagation of chaos} \label{se:global-to-local}

As we have mentioned, there are apparently not many methods for \emph{quantitatively} relating global chaos \eqref{intro:def:globalchaos} and local chaos \eqref{intro:def:localchaos}. Subadditivity inequalities like \eqref{def:intro:subadditivity} give one method for passing from global to local estimates, which our results show to be suboptimal in our  setting.
The only alternative technique we are aware of is due to the recent work \cite{hauray2014kac}.
The approach of \cite[Theorem 2.4]{hauray2014kac} starts from the triangle inequality, 
\begin{align*}
\W_1(P^n_k,\mu^{\otimes k}) &\le \W_1\bigg(P^n_k, \int L_n^{\otimes k}\,dP^n\bigg) + \W_1\bigg(\int L_n^{\otimes k}\,dP^n, \, \mu^{\otimes k}\bigg).
\end{align*}
Under a suitable moment assumption, using the same combinatorial argument underlying the Diaconis-Freedman estimate $k(k-1)/n$ mentioned above, the first term is shown to be $O(k^2/n)$.
The second term can be bounded from above by
\begin{align*}
\W_1\bigg(\int L_n^{\otimes k}\,dP^n, \, \mu^{\otimes k}\bigg) \le \int \W_1(L_n^{\otimes k},\mu^{\otimes k})\,dP^n = k \int \W_1(L_n,\mu)\,dP^n,
\end{align*}
with the first bound coming from convexity of $\W_1$ and the second from an additivity argument  \cite[Proposition 2.6(i)]{hauray2014kac}.
Ultimately, the first $O(k^2/n)$ term bound already falls short of our optimal local estimate of $\W_1(P^n_k,\mu^{\otimes n}) = O(k/n)$ (see Corollary \ref{co:static-convex}), and the second will surely be no better than $O(k / n^{(1/2) \wedge (1/d)})$ as it is governed by the mean rate of convergence of the (global) empirical measure in Wasserstein distance \cite{fournier2015rate}.

\section{Proofs of main theorems} \label{se:proofs-static}

This section is devoted to the proofs of Theorems \ref{th:main} and \ref{th:main-rev}.
In both cases, we may assume without loss of generality that $\beta=1$, noting that the general case can be recovered from the $\beta=1$ case by changing $M$ and $\transpconst$ to $M\beta^2$ and $\transpconst\beta^2$, respectively.
Both theorems make use of the following simple calculation of the logarithmic gradients of marginal densities.
Recall that $P^n_{k+1|k}(x)$ denotes the conditional law of $X_{k+1}$ given $(X_1,\ldots,X_k)=x$, under $P^n$.
For $1\le i \le k$ and a function $f$ on $(\R^d)^k$, we write $\nabla_if$ for the gradient in the $i^\text{th}$ argument.

\begin{lemma} \label{le:inv-conditional}
Let $1 \le i \le k < n$. The marginal distribution $P^n_k$ satisfies
\begin{align*}
-\nabla_i \log \frac{dP^n_k}{d\mu^{\otimes k}}(x_1,\ldots,x_k) &= \frac{1}{n-1}\sum_{j \le k, \, j \neq i}\big(\nabla_1 V(x_i,x_j)-\langle \mu,\nabla_1V(x_i,\cdot)\rangle\big) \\
	&\qquad  + \frac{n-k}{n-1}\big\langle P^n_{k+1|k}(x_1,\ldots,x_k) - \mu, \nabla_1 V(x_i,\cdot)\big\rangle .
\end{align*}
\end{lemma}
\begin{proof}
Let $f^n_k = dP^n_k/d\mu^{\otimes k}$.
From the formulas \eqref{def:Pn} and \eqref{def:mu-fixedpoint},
\begin{align*}
f^n_n(x_1,\ldots,x_k) = \frac{Z}{Z_n}\exp\bigg(-\frac{1}{n-1}\sum_{1 \le i < j \le n} \big( V(x_i,x_j) - \langle \mu, V(x_i,\cdot)\rangle\big)\bigg).
\end{align*}
Start from the identity
\begin{align*}
f^n_k(x_1,\ldots,x_k) &= \int_{(\R^d)^{n-k}}f^n_n(x_1,\ldots,x_n) \, \mu(dx_{k+1})\cdots \mu(dx_n).
\end{align*}
Since $V$ is bounded from below, 
we may exchange differentiation and integration to express $-\nabla_i \log f^n_k(x_1,\ldots,x_k)$ as the conditional expectation
\begin{align*}
\frac{\int_{(\R^d)^{n-k}} \frac{1}{n-1}\sum_{j \neq i} \big(\nabla_1 V(x_i,x_j) - \langle \mu, \nabla_1 V(x_i,\cdot)\rangle\big) f^n_n(x_1,\ldots,x_n) \, \mu(dx_{k+1})\cdots \mu(dx_n)}{f^n_k(x_1,\ldots,x_k)}.
\end{align*}
Note this also used symmetry of $V$.
The $j \le k$ terms can be pulled out of the integral, as they are not integrated. For each $j > k$, exchangeability implies
\begin{align*}
 \frac{\int_{(\R^d)^{n-k}}\nabla_1 V(x_i,x_j) f^n_n(x_1,\ldots,x_n) \, \mu(dx_{k+1})\cdots \mu(dx_n)}{f^n_k(x_1,\ldots,x_k)} = \big\langle P^n_{k+1|k}(x_1,\ldots,x_k), \nabla_1 V(x_i,\cdot)\big\rangle.
\end{align*}
This completes the proof.
\end{proof}

\subsection{Proof of Theorem \ref{th:main}}
{\ }

\noindent\textit{Step 1.} 
The first step will yield an $O(k^3/n^2)$ estimate, which we improve to $O(k^2/n^2)$ in the second step.
Let $1 \le k < n$. 
Define $J_k \ge 0$ by
\begin{align*}
J_k^2 := \frac{1}{k}I(P^n_k\,|\,\mu^{\otimes k}) = \int_{(\R^d)^k}\bigg| \nabla_1 \log \frac{dP^n_k}{d\mu^{\otimes k}}\bigg|^2 dP^n_k,
\end{align*}
where the second identity follows from exchangeability. Lemma \ref{le:inv-conditional} yields
\begin{align*}
J_k^2 = \int_{(\R^d)^k}\Bigg| & \frac{1}{n-1}\sum_{j =2}^k\big(\nabla_1 V(x_1,x_j)- \langle \mu, \nabla_1 V(x_1,\cdot)\rangle\big) \\
	& + \frac{n-k}{n-1}\langle P^n_{k+1|k}(x) - \mu,\nabla_1 V(x_1,\cdot)\rangle  \Bigg|^2\!P^n_k(dx).
\end{align*}
By the triangle inequality, we have
\begin{align}
\begin{split}
J_k \le \ & \bigg( \int_{(\R^d)^k}\bigg|\frac{1}{n-1} \sum_{j =2}^k \big( \nabla_1 V(x_1,x_j)  - \langle \mu, \nabla_1 V(x_1,\cdot)\rangle\big) \bigg|^2 P^n_k(dx)\bigg)^{1/2} \\
	& + \bigg( \int_{(\R^d)^k}\bigg|\frac{n-k}{n-1} \langle P^n_{k+1|k}(x) - \mu, \nabla_1 V(x_1,\cdot)\rangle \bigg|^2 P^n_k(dx)\bigg)^{1/2}.
	\end{split} \label{pf:static-key1}
\end{align}
Using exchangeability and the triangle inequality, the first term of \eqref{pf:static-key1} is bounded by
\begin{align*}
\frac{k-1}{n-1}\sqrt{M},
\end{align*}
where $M$ was defined in \eqref{asmp:static-moment}.
Similarly, we note here for later use that if $k=n$ then 
\begin{align}
J_n = \bigg( \int_{(\R^d)^k}\bigg|\frac{1}{n-1} \sum_{j =2}^n \big( \nabla_1 V(x_1,x_j)  - \langle \mu, \nabla_1 V(x_1,\cdot)\rangle\big) \bigg|^2 P^n_k(dx)\bigg)^{1/2} \le \sqrt{M}. \label{pf:static-Jnbound}
\end{align}
To estimate the second term of \eqref{pf:static-key1} for general $k$, let $(x,x_{k+1})$ denote a generic element of $(\R^d)^{k+1}$, and use the assumption \eqref{asmp:static-transp} to get
\begin{align*}
|\langle  P^n_{k+1|k}(x) - \mu, \nabla_1 V(x_1,\cdot)\rangle|^2 
	&\le \transpconst I( P^n_{k+1|k}(x) \,|\, \mu ) \\
	&= \transpconst \int_{\R^d} \left| \nabla_{k+1} \log \frac{dP^n_{k+1|k}(x)}{d\mu}(x_{k+1})\right|^2 \, P^n_{k+1|k}(x)(dx_{k+1}) \\
	&= \transpconst \int_{\R^d} \left| \nabla_{k+1} \log \frac{dP^n_{k+1}}{d\mu^{\otimes (k+1)}}(x,x_{k+1})\right|^2 \, P^n_{k+1|k}(x)(dx_{k+1}),
\end{align*}
where the last step above follows from the identity
\begin{align*}
\frac{dP^n_{k+1}}{d\mu^{\otimes (k+1)}}(x,x_{k+1}) &= \frac{dP^n_{k+1|k}(x)}{d\mu}(x_{k+1})\frac{dP^n_k}{d\mu^{\otimes k}}(x).
\end{align*}
Thus, discarding the factor $(n-k)/(n-1)$, the second term of \eqref{pf:static-key1} is bounded by
\begin{align*}
&\left[\transpconst \int_{(\R^d)^k} \int_{\R^d} \left| \nabla_{k+1} \log \frac{dP^n_{k+1}}{d\mu^{\otimes (k+1)}}(x,x_{k+1})\right|^2 \!\! P^n_{k+1|k}(x)(dx_{k+1})  P^n_k(dx)\right]^{1/2},
\end{align*}
which is exactly equal to $\sqrt{\transpconst } J_{k+1}$ by exchangeability. Putting it together, we deduce from \eqref{pf:static-key1} that
\begin{align*}
J_k \le \frac{k-1}{n-1}\sqrt{M} + \sqrt{\transpconst } J_{k+1}.
\end{align*}
Iterate this inequality to get
\begin{align}
J_k &\le \frac{\sqrt{M}}{n-1}\sum_{\ell=k}^{n-1} \transpconst^{\tfrac12(\ell-k)}(\ell - 1) + \transpconst^{\tfrac12(n-k)} J_n. \label{pf:Jk-iterate1}
\end{align}
Since $\transpconst < 1$, for $k \ge 2$ we have the estimate
\begin{align*}
\sum_{\ell=k}^{n-1} \transpconst^{\tfrac12(\ell-k)}(\ell - 1) &= \sum_{\ell=0}^{n-k-1} \transpconst^{\tfrac12 \ell }(\ell + k - 1) \le \sum_{\ell=0}^\infty \transpconst^{\tfrac12 \ell } (\ell+1) + (k - 2)\sum_{\ell=0}^\infty \transpconst^{\tfrac12 \ell } \\
	&= \frac{1}{(1-\sqrt{\transpconst})^2} + \frac{k-2}{1-\sqrt{\transpconst}} \\
	&= \frac{\sqrt{\transpconst} + (1-\sqrt{\transpconst})(k-1)}{(1-\sqrt{\transpconst})^2}  \\
	&\le \frac{k-1}{(1-\sqrt{\transpconst})^2}.
\end{align*}
Thus, using also \eqref{pf:static-Jnbound}, the estimate \eqref{pf:Jk-iterate1} implies
\begin{align}
J_k &\le \frac{ \sqrt{M } }{(1-\sqrt{\transpconst})^2} \frac{k-1}{n-1} + \transpconst^{\tfrac12(n-k)} \sqrt{M}. \label{pf:Jk-iterate2}
\end{align}
Finally, this yields
\begin{align}
I(P^n_k\,|\,\mu^{\otimes k}) &= kJ_k^2 \le kM\left(\frac{ 1 }{(1-\sqrt{\transpconst})^2} \frac{k-1}{n-1} + \transpconst^{\tfrac12(n-k)} \right)^2. \label{pf:static-k^3}
\end{align}
This is $O(k^3/n^2)$ instead of the desired $O(k^2/n^2)$, so we perform one more step.

{\ }

\noindent\textit{Step 2.} 
Knowing the inequality \eqref{pf:static-k^3}, we can now use it to improve the estimate of first term in \eqref{pf:static-key1} as follows. Expand the square and use exchangeability to get
\begin{align*}
& \int_{(\R^d)^k}  \Bigg|\frac{1}{n-1} \sum_{j =2}^k \big( \nabla_1 V(x_1,x_j)  - \langle \mu, \nabla_1 V(x_1,\cdot)\rangle\big) \Bigg|^2 P^n_k(dx) \\
	& \qquad = \frac{k-1}{(n-1)^2}M + \frac{(k-1)(k-2)}{(n-1)^2}R,
\end{align*}
where we define
\begin{align*}
R &:= \int_{(\R^d)^3} \big( \nabla_1 V(x_1,x_2)  - \langle \mu, \nabla_1 V(x_1,\cdot)\rangle\big) \cdot \big( \nabla_1 V(x_1,x_3)  - \langle \mu, \nabla_1 V(x_1,\cdot)\rangle\big) \, P^n_3(dx).
\end{align*}
Condition on $(x_1,x_2)$ and use Cauchy-Schwarz to get
\begin{align*}
R &= \int_{(\R^d)^2}  \big( \nabla_1 V(x_1,x_2)  - \langle \mu, \nabla_1 V(x_1,\cdot)\rangle\big) \cdot \langle P^n_{3|2}(x_1,x_2) -  \mu, \nabla_1 V(x_1,\cdot)\rangle \, P^n_2(dx_1,dx_2) \\
	&\le \bigg(M\int_{(\R^d)^2} |\langle P^n_{3|2}(x_1,x_2) -  \mu, \nabla_1 V(x_1,\cdot)\rangle|^2 \, P^n_2(dx_1,dx_2)\bigg)^{1/2}.
\end{align*}
Use the assumption \eqref{asmp:static-transp} to bound the integrand by $\transpconst I(P^n_{3|2}(x_1,x_2)\,|\,\mu)$,  which yields
\begin{align*}
R &  \le \bigg(\transpconst M \int_{(\R^d)^2} \int_{\R^d} \bigg|\nabla \log \frac{dP^n_{3|2}(x_1,x_2)}{d\mu}(x_3) \bigg|^2 P^n_{3|2}(x_1,x_2)(dx_3) \, P^n_2(dx_1,dx_2)\bigg)^{1/2} \\
&  = \bigg(\transpconst M \int_{(\R^d)^3} \bigg|\nabla \log \frac{dP^n_3}{d\mu^{\otimes 3}}\bigg|^2  \,dP^n_3\bigg)^{1/2} \\
&  = J_3\sqrt{\transpconst M}.
\end{align*}
Applying \eqref{pf:Jk-iterate2} with $k=3$,
\begin{align*}
R  &\le \frac{ M\sqrt{\transpconst} }{(1-\sqrt{\transpconst})^2} \frac{2}{n-1} + \transpconst^{\tfrac12(n-2)} M.
\end{align*}
Note since $n \ge 3$ that $(n-1) \le 2(n-2)$. Use also $\sup_{x > 0} x \transpconst^{x/2} = 2/e\log(1/\transpconst)$ to get
\begin{align*}
R  &\le \frac{MC_1}{n-1}, \ \ \ \text{ where } \ \ \ C_1 :=  \frac{ 2 \sqrt{\transpconst } }{(1-\sqrt{\transpconst})^2} + \frac{4}{e \log(1/\transpconst)}
\end{align*}
Putting it together, we have thus improved the bound on the first term of \eqref{pf:static-key1} to 
\begin{align*}
& \int_{(\R^d)^k}  \Bigg|\frac{1}{n-1} \sum_{j =2}^k \big( \nabla_1 V(x_1,x_j)  - \langle \mu, \nabla_1 V(x_1,\cdot)\rangle\big) \Bigg|^2 P^n_k(dx) \\
	& \qquad \le \frac{k-1}{(n-1)^2}M + \frac{(k-1)(k-2)}{(n-1)^3} M C_1 \\
	& \qquad \le \frac{k-1}{(n-1)^2}M(1+C_1).
\end{align*}
The bound on the second term of \eqref{pf:static-key1} remains the same as before, and we thus get
\begin{align*}
J_k \le \frac{\sqrt{k-1}}{n-1}\sqrt{M(1+C_1)} + \sqrt{\transpconst } J_{k+1}.   
\end{align*}
Iterate this inequality to get 
\begin{align}
J_k &\le \frac{\sqrt{M(1+C_1)}}{n-1}\sum_{\ell=k}^{n-1} \transpconst^{\tfrac12(\ell-k)}\sqrt{\ell - 1} + \transpconst^{\tfrac12(n-k)} J_n. \label{pf:Jk-iterate1-2}
\end{align}
Estimate the summation by noting that $\transpconst < 1$ implies
\begin{align*}
\sum_{\ell=k}^{n-1} \transpconst^{\tfrac12(\ell-k)}\sqrt{\ell - 1} &= \sum_{\ell=0}^{n-k-1} \transpconst^{\tfrac12 \ell} \sqrt{\ell + k - 1} \le \sum_{\ell=0}^\infty \transpconst^{\tfrac12 \ell } \sqrt{\ell} + \sqrt{k - 1}\sum_{\ell=0}^\infty \transpconst^{\tfrac12 \ell }.
\end{align*}
This is bounded by $C_2\sqrt{k - 1}$, where we define $C_2 := 2\sum_{\ell=0}^\infty \transpconst^{\tfrac12 \ell }\sqrt{\ell}$.
We deduce from \eqref{pf:Jk-iterate1-2} and also \eqref{pf:static-Jnbound} that
\begin{align*}
J_k &\le C_2\sqrt{M(1+C_1)}\frac{\sqrt{k-1}}{n-1} + \transpconst^{\tfrac12(n-k)} \sqrt{M}.
\end{align*}
Finally, this yields
\begin{align}
I(P^n_k\,|\,\mu^{\otimes k}) &= kJ_k^2 \le k M  \left( C_2^2(1 + C_1)  \frac{\sqrt{k-1}}{n-1} + \transpconst^{\tfrac12(n-k)} \right)^2. \label{pf:static-k^2}
\end{align}

The constant $C_2^2(1+C_1)$ can be simplified by noting that
\begin{align*}
1 + C_1 = \frac{ 1 + \transpconst }{(1-\sqrt{\transpconst})^2} + \frac{4}{e \log(1/\transpconst)}.
\end{align*}
To estimate $C_2$, note for $0 < y <1$ that
\begin{align*}
\sum_{\ell=0}^\infty \sqrt{\ell} y^\ell &\le \int_0^\infty \sqrt{x} y^x \,dx =  2\int_0^\infty z^2 e^{z^2\log y} \,dz = \sqrt{\pi/ \log(1/y)}.
\end{align*}
Apply this with $y=\sqrt{\transpconst}$ to get $C_2 \le 2\sqrt{ 2\pi/ \log (1/\transpconst)}$. Hence,
\begin{align*}
C_2^2(1+C_1) &\le  \frac{ 8\pi  }{ \log (1/\transpconst) } \left( \frac{  1 + \transpconst }{(1-\sqrt{\transpconst})^2  } + \frac{4}{e \log (1/\transpconst)}\right).
\end{align*}
{\ } \vskip-1cm \hfill\qedsymbol

\subsection{Proof of Corollary \ref{co:static-bounded}} \label{se:proof:bded}
We first borrow an argument from \cite[Theorem 3.1]{guillin2009transportation}:
Let $\nu \in \P(\R^d)$ with $\nu\ll\mu$ and $f=d\nu/d\mu$. By Cauchy-Schwarz,
\begin{align*}
\|\nu-\mu\|_{\mathrm{TV}}^2 &= \bigg(\int |f-1|\,d\mu\bigg)^2 = \bigg(\int |(\sqrt{f}+1)(\sqrt{f}-1)|\,d\mu\bigg)^2 \\
	&\le \int (\sqrt{f}+1)^2\,d\mu \int (\sqrt{f}-1)^2\,d\mu = 4\Var_\mu(\sqrt{f}).
\end{align*}
Apply the assumed Poincar\'e inequality \eqref{asmp:co:static-poincare} to get
\begin{align}
\|\nu-\mu\|_{\mathrm{TV}}^2 \le c_\mu I(\nu\,|\,\mu), \qquad \forall \nu \in \P(\R^d). \label{ineq:Fisher-Pinsker}
\end{align}
Using the boundedness assumption (1), we deduce
\begin{align*}
|\langle \mu - \nu, \nabla_1 V(x,\cdot)\rangle|^2 \le  L^2 c_\mu I(\nu\,|\,\mu), \ \ \forall x \in \R^d, \ \nu \in \P(\R^d).
\end{align*}
This shows that the assumption \eqref{asmp:static-transp} of Theorem \ref{th:main} holds, and clearly the assumption \eqref{asmp:static-moment} also holds with $M \le 2L^2$. We may thus apply Theorem \ref{th:main}. 
Lastly, note that the inequality \eqref{ineq:Fisher-Pinsker} above is precisely the case $k=1$ of the claim \eqref{ineq:hamming1}. Tensorize this transport inequality as in \cite[Corollary 2.13]{guillin2009transportation} to get the claim.
\hfill\qedsymbol

\begin{remark}
Corollary \ref{co:static-bounded} could likely be generalized, using a more involved inequality such as \cite[Theorem 5.1]{guillin2009transportation} instead of \eqref{ineq:Fisher-Pinsker} to deduce a \emph{weighted} Pinsker-type inequality from the assumed Poincar\'e inequality. This would allow us to relax the assumption of boundedness of $\nabla_1V$, but for the sake of brevity we do not pursue this.
\end{remark}

\subsection{Proof of Theorem \ref{th:main-rev}}

We again assume without loss of generality that $\beta=1$.
Abbreviate $H_k:= H( \mu^{\otimes k} \,|\, P^n_k)$ and $I_k:= I( \mu^{\otimes k} \,|\, P^n_k)$.
First, expand the square, using the definition of $M$ from \eqref{asmp:static-moment-rev} and exchangeability to compute
\begin{align*}
I_n &= \sum_{i=1}^n\int_{(\R^d)^n}\Bigg|\frac{1}{n-1} \sum_{ j \neq i}\big(\nabla_1 V(x_i,x_j) - \langle \mu, \nabla_1 V(x_i,\cdot)\rangle\big) \Bigg|^2 \mu^{\otimes n}(dx) = \frac{n}{n-1}M.
\end{align*}
Using the assumption \eqref{asmp:static-LSI-rev}, this implies, for each $k=1,\ldots,n$, that
\begin{align}
H_k &\le H_n \le \eta I_n \le \frac{n}{n-1}\eta M. \label{pf:static-rev-apriori1}
\end{align}
Now fix $n > k \ge 1$. 
By Lemma \ref{le:inv-conditional}, 
\begin{align*}
I_k \le \sum_{i=1}^k\!\int_{(\R^d)^k} \bigg| & \frac{1}{n-1}\!\sum_{j \le k, \, j \neq i} \big(\nabla_1 V(x_i,x_j)  - \langle \mu, \nabla_1 V(x_i,\cdot)\rangle\big) \\
 & + \frac{n-k}{n-1}\langle P^n_{k+1|k}(x)-\mu, \nabla_1 V(x_i,\cdot)\rangle \bigg|^2 \mu^{\otimes k}(dx).
\end{align*}
Let $\epsilon > 0$. Use exchangeability  and the inequality $(a+b)^2 \le (1+\epsilon)a^2 + (1+\epsilon^{-1})b^2$ to get
\begin{align*}
I_k \le \ & (1+\epsilon^{-1})k \int_{(\R^d)^k}\bigg|\frac{1}{n-1}\!\sum_{j \le k, \, j \neq i}\big(\nabla_1 V(x_i,x_j - \langle \mu, \nabla_1 V(x_i,\cdot)\rangle \big) \bigg|^2 \mu^{\otimes k}(dx) \\
	&+ (1+\epsilon)k\int_{(\R^d)^k}\bigg| \frac{n-k}{n-1} \big\langle P^n_{k+1|k}(x) - \mu, \nabla_1 V(x_i,\cdot)\big\rangle \bigg|^2 \mu^{\otimes k}(dx).
\end{align*}
The first term, using independence, is equal to
\begin{align*}
(1+\epsilon^{-1}) \frac{k(k-1)}{(n-1)^2}.
\end{align*}
The second term we bound by first discarding the term $(n-k)/(n-1)$. Then use the assumption \eqref{asmp:static-transp-rev} followed by the chain rule for relative entropy \cite[Theorem 2.6]{budhiraja-dupuis} to get
\begin{align*}
\int_{(\R^d)^k}\Big|\big\langle P^n_{k+1|k}(x) - \mu, \nabla_1 V(x_i,\cdot)\big\rangle\Big|^2 \mu^{\otimes k}(dx) &\le \transpconst  \int_{(\R^d)^k} H\big( \mu \,|\, P^n_{k+1|k}(x) \big) \, \mu^{\otimes k}(dx) \\
	&= \transpconst (H_{k+1}-H_k).
\end{align*}
Putting it together, and using the assumption \eqref{asmp:static-LSI-rev}, we find
\begin{align*}
H_k &\le \eta I_k \le \eta  (1+\epsilon^{-1}) \frac{k(k-1)}{(n-1)^2} + \eta \transpconst (1+\epsilon)k (H_{k+1}-H_k).
\end{align*}
Rearrange this to get
\begin{align*}
H_k &\le a_k + c_kH_{k+1}, \ \text{ where } \  a_k := \frac{ \eta  (1+\epsilon^{-1}) \frac{k(k-1)}{(n-1)^2}}{ 1 +  \eta \transpconst(1+\epsilon)k }, \quad c_k := \frac{ \eta \transpconst (1+\epsilon)k}{ 1 +  \eta \transpconst (1+\epsilon)k }.
\end{align*}
Iterate this inequality to get
\begin{align}
H_k &\le \sum_{\ell=k}^{n-1} \bigg(\prod_{j=k}^\ell c_j\bigg) \frac{a_\ell}{c_\ell} +  \bigg(\prod_{j=k}^{n-1} c_j\bigg)H_{n}. \label{pf:rev-iter1}
\end{align}
Abbreviate $\alpha := \eta \transpconst (1+\epsilon)$, so that $c_j = \alpha j/(1+\alpha j)$.
We then have the estimate
\begin{align*}
\prod_{j=k}^\ell c_j &= \prod_{j=k}^{\ell}\frac{\alpha j}{1 + \alpha  j} = \exp\sum_{j=k}^\ell \log \bigg( 1 - \frac{1}{1+ \alpha  j}\bigg) \le \exp\bigg( - \sum_{j=k}^\ell \frac{1}{1 + \alpha j } \bigg)  \\
	&\le \exp\bigg( - \int_{k}^{\ell} \frac{1}{1 + \alpha u }du \bigg) = \exp\bigg( - \frac{1}{\alpha } \log\frac{1+\alpha \ell}{1+\alpha k}\bigg) \\
	&= \bigg(\frac{1 + \alpha k}{1 + \alpha \ell}\bigg)^{\tfrac{1}{\alpha }}.
\end{align*}
Thus, using also \eqref{pf:static-rev-apriori1}, and noting that $\frac{a_\ell}{c_\ell} = \frac{\ell-1}{\transpconst  \epsilon (n-1)^2}$, we deduce from \eqref{pf:rev-iter1} that
\begin{align*}
H_k &\le \frac{1 }{\transpconst \epsilon (n-1)^2}\sum_{\ell=k}^{n-1} (\ell-1)\left(\frac{1 + \alpha k}{1 + \alpha \ell}\right)^{\tfrac{1}{\alpha }} +  \eta M \frac{n}{n-1}\left(\frac{1 + \alpha k}{1 + \alpha (n-1)}\right)^{\tfrac{1}{\alpha }}.
\end{align*}
We lastly estimate the sum
\begin{align*}
\sum_{\ell=k}^{n-1} (\ell-1)(1 + \alpha \ell)^{-\tfrac{1}{\alpha }} &\le \frac{1}{\alpha}\sum_{\ell=k}^{n-1} (1 + \alpha \ell)^{1-\tfrac{1}{\alpha }}.
\end{align*}

\textit{Case 1.}
Suppose $\alpha > 1/2$. Then
\begin{align*}
\sum_{\ell=k}^{n-1} (1 + \alpha \ell)^{1-\tfrac{1}{\alpha }} &\le \int_{0}^{n}(1 + \alpha x)^{1-\tfrac{1}{\alpha }} \, dx = \frac{1}{2\alpha - 1 }(1 + \alpha n)^{2-\tfrac{1}{\alpha }}.
\end{align*}
and we get
\begin{align*}
H_k &\le \frac{ (1 + \alpha k)^{\tfrac{1}{\alpha}} }{(2\alpha-1) \alpha \transpconst \epsilon (n-1)^2}(1 + \alpha n)^{2-\tfrac{1}{\alpha }} +  \eta M \frac{n}{n-1}\left(\frac{1 + \alpha k}{1 + \alpha (n-1)}\right)^{\tfrac{1}{\alpha }} \\
	&\le \left(\frac{ 1 }{(2\alpha-1) \alpha \transpconst \epsilon } \left( \frac{1 + \alpha n}{n-1}\right)^2 +  \eta M \frac{n}{n-1} \right)\left(\frac{1 + \alpha k}{1 + \alpha (n-1)}\right)^{\tfrac{1}{\alpha }} \\
	&\le \left(\frac{ (1+2\alpha)^2}{(2\alpha -1) \alpha \transpconst \epsilon }  +  2\eta M \right)\left(\frac{1 + \alpha k}{1 + \alpha (n-1)}\right)^{\tfrac{1}{\alpha }},
\end{align*}
where the last step used $1+\alpha n \le (1+2\alpha)(n-1)$ for $n \ge 2$, and $n\le 2(n-1)$.

\textit{Case 2.} Now suppose $1/2 > \alpha > 0$. Use the estimate
\begin{align*}
\sum_{\ell=k}^{n-1} (1 + \alpha \ell)^{1-\tfrac{1}{\alpha }} &\le \int_{k-1}^\infty(1 + \alpha x)^{1-\tfrac{1}{\alpha }} \, dx = \frac{1}{1-2\alpha}(1+\alpha(k-1))^{2-\tfrac{1}{\alpha }},
\end{align*}
and we get
\begin{align*}
H_k &\le \frac{ (1 + \alpha k)^{\tfrac{1}{\alpha}} }{(1-2\alpha) \alpha  \transpconst \epsilon (n-1)^2}(1 + \alpha (k-1))^{2-\tfrac{1}{\alpha }} +  \eta  M \frac{n}{n-1}\left(\frac{1 + \alpha k}{1 + \alpha (n-1)}\right)^{\tfrac{1}{\alpha }} \\
	&\le \frac{ (1+\alpha)^{\frac{1}{\alpha}} }{(1-2\alpha) \alpha \transpconst \epsilon  }\left(\frac{1+\alpha(k-1)}{n-1}\right)^2 + 2\eta  M \left(\frac{1 + \alpha k}{1 + \alpha (n-1)}\right)^{\tfrac{1}{\alpha }},
\end{align*}
where we used $1+\alpha k \le (1+\alpha(k-1))(1+\alpha)$ for $k \ge 1$. Noting that $1/\alpha > 2$ and $1+ \alpha(n-1) \le n-1$ for $\alpha < 1/2$ and $n \ge 2$, we get
\begin{align*}
H_k &\le \left(\frac{ (1+\alpha)^{\frac{1}{\alpha}} }{(1-2\alpha)\alpha  \transpconst \epsilon  } + 2\eta  M \right)\left(\frac{1 + \alpha k}{1 + \alpha (n-1)}\right)^2.
\end{align*}

\section{The case of convex potentials} \label{se:proofs-BE}

This section is devote to the proofs of Corollaries \ref{co:static-convex} and \ref{co:static-convex-rev}. In each case, existence and uniqueness of $\mu \in \P(\R^d)$ satisfying \eqref{def:mu-fixedpoint} follows from Assumption \ref{assumption:B}, by \cite[Theorem 2.1]{carrillo2003kinetic}.
We first collect some useful and well known facts that will be used in each proof.

\subsection{The Bakry-\'Emery curvature condition} \label{se:BakryEmery}

This short section summarizes a number of classical facts about strongly log-concave probability measures; see \cite{saumard2014log} for an overview and references.
We say a function on $\R^k$  is $\kappa$-convex if its Hessian is pointwise bounded from below by $\kappa I$, in semidefinite order.
Let $\nu$ be a probability measure on $\R^k$ with strictly positive density, also denoted $\nu$, such that $-\log \nu$ is $\kappa$-convex. Then $\nu$ satisfies the log-Sobolev inequality
\begin{align*}
H(f\nu\,|\,\nu) &\le \frac{2}{\kappa}\int |\nabla \sqrt{f}|^2 \,d\nu = \frac{1}{2\kappa}\int|\nabla \log f|^2 f\,d\nu = \frac{1}{2\kappa}I(f\nu\,|\,\nu).
\end{align*}
This is a famous result of Bakry-\'Emery \cite{bakryemery}, and see \cite[Corollary 5.7.2]{bakryGentilLedoux} or \cite[Corollary 7.3]{gozlan-leonard} for English references.
It was later shown by Otto-Villani \cite{ottovillani} that this implies the quadratic transport inequality
\begin{align*}
\W_2^2(\nu,\cdot) \le \frac{2}{\kappa}H(\cdot\,|\,\nu).
\end{align*}
We also have the Poincar\'e inequality (see \cite[Section 7]{ottovillani})
\begin{align}
\Var_\nu(f) := \int (f-\langle\nu,f\rangle)^2\,d\nu &\le \frac{1}{\kappa}\int |\nabla f|^2\,d\nu. \label{eq:BE-Poincare}
\end{align}

We will also make use of the fact that strong log-concavity is preserved under conditioning and marginalization. The first of these properties is immediate, and the second is due to Brascamp-Lieb \cite[Theorem 4.3]{brascamp2002extensions} (see also \cite[Theorem 3.8]{saumard2014log}):

\begin{lemma} \label{le:logconcave} 
Suppose $\nu(x,y)$ is a strictly positive probability density function on $\R^{d+d'}$. Define the conditional and marginal densities
\begin{align*}
\nu_x(y) &= \frac{\nu(x,y)}{\int_{\R^{d'}}\nu(x,y') dy'}, \quad\qquad \widetilde\nu(x) = \int_{\R^{d'}} \nu(x,y)\,dy.
\end{align*} 
Let $\kappa \ge 0$.
If $-\log\nu$ is $\kappa$-convex on $\R^{d+d'}$, then $-\log\nu_x(\cdot)$ is $\kappa$-convex on $\R^{d'}$ for each $x \in \R^d$, and $-\log\widetilde\nu$ is $\kappa$-convex on $\R^d$.
\end{lemma}

\subsection{Proofs of Corollaries \ref{co:static-convex} and \ref{co:static-convex-rev}}

We now specialize to the context of $P^n$ and $\mu$ as defined in Section \ref{se:mainresults}, under Assumption \ref{assumption:B}.
Identifying the measures $P^n$ and $\mu$ with their densities, Assumption \ref{assumption:B} implies that $-\log P^n$ and $-\log \mu$ are both $\beta \kappa$-convex. 
The latter is straightforward from the assumed convexity of $V$ and $U$, and the former requires a calculation for which we refer to \cite[Lemma 3.5]{malrieu2001logarithmic} or \cite[Proposition 3.1]{malrieu2003convergence}.
As in Section \ref{se:BakryEmery}, we deduce that $\mu$ satisfies the log-Sobolev and quadratic transport inequalities
\begin{align}
\W_2^2(\mu,\cdot) \le \frac{2}{\beta \kappa}H(\cdot \, |\, \mu) \le \frac{1}{\beta^2 \kappa^2}I(\cdot\,|\,\mu). \label{eq:BE-W2H}
\end{align}
The same inequalities hold with $P^n$ in place of $\mu$.

{\ }

\noindent\textbf{Proof of Corollary \ref{co:static-convex}.}
Recall the definition of the constant $L$ from Assumption \ref{assumption:B}. Note for any unit vector $u \in \R^d$ that the function $x \mapsto u \cdot \nabla V(x)$ is $L$-Lipschitz since $0 \le \nabla^2 V \le LI$. Hence, using Kantorovich duality, for each $x \in \R^d$ and $\nu \in \P(\R^d)$ we have
\begin{align*}
\left|\langle \nu - \mu, \nabla V(x-\cdot)\rangle\right|^2 &= \sup_{u \in \R^d, \, |u|=1} \langle \nu - \mu, u \cdot \nabla V(x-\cdot)\rangle^2 \le L^2 \W_1^2(\mu,\nu).
\end{align*}
Combining this with the inequalities \eqref{eq:BE-W2H}, we deduce that condition (2) of Theorem \ref{th:main} holds with $\transpconst = (L/\beta\kappa)^2$:
\begin{align}
\left|\langle \nu - \mu, \nabla V(x-\cdot)\rangle\right|^2 &\le L^2 \W_2^2(\mu,\nu) \le \frac{ L^2}{\beta^2 \kappa^2} I(\nu\,|\,\mu). \label{pf:BE-transp}
\end{align}
It remains finally to estimate the constant
\begin{align*}
M := \int_{(\R^d)^2} \left|  \nabla V(x_1-x_2) - \langle \mu,\nabla V(x_1-\cdot)\rangle \right|^2 \, P^n_2 (dx_1,dx_2).
\end{align*}
By Jensen's inequality and Assumption \ref{assumption:B},
\begin{align*}
M &\le \int_{\R^d \times \R^d} \int_{\R^d} \left| \nabla V(x_1-x_2) - \nabla V(x_1-x_3)\right|^2\, \mu(dx_3)\, P^n_2(dx_1,dx_2) \\
	&\le L^2 \int_{\R^d} \int_{\R^d} | x-y |^2\, \mu(dy)\, P^n_1(dx) \\
	&\le 2L^2 \left( \int_{\R^d} |x|^2\,\mu(dx) + \int_{\R^d} |x|^2\,P^n_1(dx)\right).
\end{align*}
Noting that $U$ and $V$ are even functions, we deduce that $P^n$ and $\mu$ have mean zero. The Poincar\'e inequality \eqref{eq:BE-Poincare}, applied to the coordinate functions $f(x) = x_i$ for $i=1,\ldots,d$, 
then shows that $\int_{\R^d}|x|^2\,P^n_1(dx)$ and $\int_{\R^d}|x|^2\,\mu(dx)$ are both bounded by $d/\beta\kappa$. Thus $M \le  \frac{L^2d}{\beta\kappa}$.
Theorem \ref{th:main} now applies, noting that $\beta\sqrt{\transpconst}=L/\kappa$.  \hfill \qedsymbol

{\ }

\noindent\textbf{Proof of Corollary \ref{co:static-convex-rev}.}
We check the hypotheses of Theorem \ref{th:main-rev}. Condition (1) follows exactly as in the proof of Corollary \ref{co:static-convex}, with $M \le  \frac{L^2d}{\beta\kappa}$.
To check the log-Sobolev inequality in condition (3), we use Lemma \ref{le:logconcave} to deduce that $-\log P^n_k$ is $\beta\kappa$-convex, and thus (3) holds with $\eta=1/2\beta\kappa$.
To check condition (2), first note that $-\log P^n_{k+1|k}(x_1,\ldots,x_k)(x_{k+1})$ is $\beta\kappa$-convex in $x_{k+1}$ for each $x_1,\ldots,x_k$, by Lemma \ref{le:logconcave}.
Thus, as in Section \ref{se:BakryEmery}, we deduce the quadratic transport inequality
\begin{align*}
\W_2^2(\cdot,P^n_{k+1|k}(x_1,\ldots,x_k)) \le \frac{2}{\beta\kappa}  H(\cdot\,|\,P^n_{k+1|k}(x_1,\ldots,x_k)),
\end{align*}
for $1 \le k < n$ and $x_1,\ldots,x_k \in \R^d$.
Arguing as in \eqref{pf:BE-transp}, we get
\begin{align*}
|\langle \mu - P^n_{k+1|k}(x_1,\ldots,x_k), \nabla V(x_1 - \cdot)\rangle|^2 \le  \frac{2L^2}{\beta \kappa} H(\mu\,|\,P^n_{k+1|k}(x_1,\ldots,x_k)).
\end{align*}
This shows condition (2) of Theorem \ref{th:main-rev} with $\transpconst = 2L^2/\beta \kappa$. 

We now have all of the ingredients we need to apply Theorem  \ref{th:main-rev}.
Let $\epsilon > 0$ and set
\begin{align*}
\alpha := \eta\transpconst\beta^2(1+\epsilon) = (1+\epsilon)(L/\kappa)^2.
\end{align*}
Note that $\transpconst\beta^2=2L^2\beta/\kappa = 2\beta\kappa\alpha/(1+\epsilon)$ and $2\eta M \beta^2 \le L^2d/\kappa^2=\alpha d/(1+\epsilon) \le \alpha d$.
Apply Theorem  \ref{th:main-rev} with these substitutions to complete the proof. \hfill \qedsymbol

\section{The Gaussian case} \label{se:gaussian}

This section documents Example \ref{ex:gaussian-static}, which illustrates that the $O((k/n)^2)$ estimate from Theorem \ref{th:main} cannot be improved.
Throughout the section, $J_n$ denotes the $n \times n$ matrix of all ones, and $I_n$ is the identity matrix.

Fix $a,b > 0$.
Consider the Gaussian probability measure on $\R^n$ defined by
\begin{align}
P^n(dx) &= \frac{1}{Z_n}\exp\bigg(  - \frac{a}{2}\sum_{i=1}^n x_i^2  - \frac{b}{2(n-1)}\sum_{1 \le i < j \le n}^n (x_i-x_j)^2\bigg) dx. \label{def:Pn-gaussian}
\end{align}
We will show that, if $n\to\infty$ and $k\to k^* \in \N \cup \{\infty\}$ with $k=o(n)$, then
\begin{align*}
(n/k)^2\W_2^2(P^n_k,\mu^{\otimes k}) &\to \frac{b^2}{4a^2(a+b)^3}\bigg( \frac{a^2}{k^*} +  \bigg(\bigg(1-\frac{1}{k^*}\bigg)a+b \bigg)^2\bigg) > 0.
\end{align*}
To see this, we begin by rewriting the exponent in \eqref{def:Pn-gaussian} as
\begin{align*}
- \frac12\bigg(a + \frac{bn}{n-1}\bigg) \sum_{i=1}^n x_i^2  + \frac{b}{2(n-1)}\sum_{i,j=1}^n x_ix_j = -\frac12 x^\top \Sigma_n^{-1} x,
\end{align*}
where $\Sigma_n^{-1} := \left(a + \frac{bn}{n-1}\right)I_n -\frac{b}{n-1}J_n$.
In other words, $P^n$ is the centered Gaussian measure with covariance matrix
\begin{align*}
\Sigma_n &= d_n(I_n + c_n J_n), \ \ \text{ where } \ \ d_n := \frac{1}{a + \frac{bn}{n-1}}, \ \text{ and } \ c_n :=  \frac{b}{a(n-1)} .
\end{align*}
The $k$-dimensional marginal $P^n_k$ is then a centered Gaussian with covariance matrix
\begin{align*}
\Sigma_{n,k} &:= d_n(I_k + c_n J_k). 
\end{align*}
Note since $c_n \to 0$ and $d_n\to 1/(a+b)$ that 
\begin{align*}
\lim_{n\to\infty}\Sigma_{n,k} &= \frac{1}{a+b}I_k, \qquad \forall k \in \N.
\end{align*}
This implies that $P^n_k$ converges weakly to $\mu^{\otimes k}$ as $n\to\infty$, for each $k \in\N$, where $\mu$ is defined as the 1-dimensional centered Gaussian measure with variance $1/(a+b)$.

As shown in \cite{dowson1982frechet}, the quadratic Wasserstein distance between two centered Gaussian measures $\gamma_1$ and $\gamma_2$ with \emph{commuting} covariance matrices $\Sigma_1$ and $\Sigma_2$ is precisely
\begin{align}
\W_2^2(\gamma_1,\gamma_2) &= \tr\big((\Sigma_1^{1/2} - \Sigma_2^{1/2})^2\big).  \label{eq:W2-gaussian}
\end{align}
For $k\in \N$, suppose $\Sigma_i = a_iI_k + b_iJ_k$ for some $a_i > 0$ and $b_i \in \R$. Note that $\Sigma_i$ has eigenvalues $a_i$ and $a_i+b_ik$, with respective multiplicities $k-1$ and $1$. Apply \eqref{eq:W2-gaussian} along with a simultaneous diagonalization of $\Sigma_1$ and $\Sigma_2$ to find
\begin{align}
\W_2^2(\gamma_1,\gamma_2) &= (k-1)\left(a_1^{1/2} - a_2^{1/2}\right)^2 + \left((a_1+b_1k)^{1/2} - (a_2+b_2k)^{1/2}\right)^2. \label{eq:W2-gaussian-ourcase}
\end{align}
Apply this in our context, with $a_1=d_n$, $a_2=1/(a+b)$, $b_1=d_nc_n$, and $b_2=0$, to get
\begin{align}
\begin{split}
\W_2^2(P^n_k,\mu^{\otimes k}) &=(k-1)(a+b)^{-1}(d_n^{1/2}(a+b)^{1/2} - 1)^2  \\
	&\qquad +  (a+b)^{-1}\left(d_n^{1/2}(a+b)^{1/2}(1 + kc_n)^{1/2} -  1\right)^2.
\end{split} \label{pf:gauss-id1}
\end{align}
Computing derivatives shows that 
\begin{align*}
d_n^{1/2}(a+b)^{1/2} = \sqrt{\frac{a+b}{a+b+\frac{b}{n-1}}} &= 1 - \frac{1}{2(a+b)}\frac{b}{n-1} + o(1/n).
\end{align*}
As $n\to\infty$ and $k\to k^* \in \N \cup \{\infty\}$, this implies
\begin{align}
\left(\frac{n}{k}\right)^2(k-1)(d_n^{1/2}(a+b)^{1/2}-1)^2 &\to \frac{b^2}{4(a+b)^2k^*}. \label{pf:gauss-id2}
\end{align}
Moreover, noting that $(1+x)^{1/2} = 1 + x/2 + o(x)$ as $x\to 0$ and  $kc_n=O(k/n)$, we have
\begin{align*}
d_n^{1/2}(a+b)^{1/2}(1 + kc_n)^{1/2} &= 1 - \frac{b}{2(a+b)}\frac{1}{n-1} + \frac12 k c_n + o(k/n),
\end{align*}
which (using $nc_n \to b/a$) leads to
\begin{align}
\frac{n}{k}\left(d_n^{1/2}(a+b)^{1/2}(1 + kc_n)^{1/2} - 1\right) &\to \frac{b}{2a} - \frac{b}{2k^*(a+b)}. \label{pf:gauss-id3}
\end{align}
Finally, plug \eqref{pf:gauss-id2} and \eqref{pf:gauss-id3} into \eqref{pf:gauss-id1} to conclude that
\begin{align*}
\left(\frac{n}{k}\right)^2\W_2^2(P^n_k,\mu^{\otimes k}) &\to \frac{b^2}{4(a+b)^3 k^*} + \frac{1}{a+b}\bigg( \frac{b}{2a} - \frac{b}{2k^*(a+b)}\bigg)^2.
\end{align*}

\bibliographystyle{amsplain}
\bibliography{biblio}

\end{document}